\documentclass[titlepage,12pt]{article} 
\usepackage{hyperref}
\usepackage{amssymb,amsthm,amsmath} 
\usepackage{mathrsfs}
\usepackage[a4paper]{geometry}
\usepackage{datetime2}
\usepackage[utf8]{inputenc}
\usepackage[italian,english]{babel}

\date{}

\newcommand{\re}{\mathbb{R}}

\newcommand{\Span}{\operatorname{Span}}

\newtheorem{thm}{Theorem}[section]

\newtheorem{rmk}[thm]{Remark}
\newtheorem{prop}[thm]{Proposition}
\newtheorem{defn}[thm]{Definition}

\newtheorem{lemma}[thm]{Lemma}

\title{A road map to the blow-up for a Kirchhoff equation with external force}

\author{Marina Ghisi\vspace{1ex}\\ 
{\normalsize Università degli Studi di Pisa} \\
{\normalsize Dipartimento di Matematica}\\ 
{\normalsize PISA (Italy)}\\
{\normalsize e-mail: \texttt{marina.ghisi@unipi.it}}
\and
Massimo Gobbino\vspace{1ex}\\ 
{\normalsize Università degli Studi di Pisa} \\
{\normalsize Dipartimento di Matematica}\\ 
{\normalsize PISA (Italy)}\\  
{\normalsize e-mail: \texttt{massimo.gobbino@unipi.it}}
}

\begin{document}
\maketitle

\begin{abstract}

It is well-known that the classical hyperbolic Kirchhoff equation admits infinitely many simple modes, namely time-periodic solutions with only one Fourier component in the space variables. 

In this paper we assume that, for a suitable choice of the nonlinearity, there exists a heteroclinic connection between two simple modes with different frequencies. Under this assumption, we cook up a forced Kirchhoff equation that admits a solution that blows-up in finite time, despite the regularity and boundedness of the forcing term.

The forcing term can be chosen with the maximal regularity that prevents the application of the classical global existence results in analytic and quasi-analytic classes.

\vspace{6ex}

\noindent{\bf Mathematics Subject Classification 2020 (MSC2020):} 
35B44, 37J46, 35L90, 35L72.

		
\vspace{6ex}


\end{abstract}

 
\section{Introduction}

Let $H$ be a real Hilbert space, and let $A$ be a positive self-adjoint operator on $H$ with dense domain $D(A)$. Let $m:[0,+\infty)\to[0,+\infty)$ and $f:[0,+\infty)\to H$ be two continuous functions. In this paper we consider the forced evolution equation 
\begin{equation}
u''(t)+m\left(|A^{1/2}u(t)|^{2}\right)Au(t)=f(t)
\label{K:eqn}
\end{equation}
with initial data
\begin{equation}
u(0)=u_{0},
\qquad\qquad
u'(0)=u_{1}.
\label{K:data}
\end{equation}

Equation (\ref{K:eqn}) is an abstract version of the hyperbolic partial differential equation introduced by  G.~Kirchhoff in the celebrated monograph~\cite[Section~29.7]{Kirchhoff} as a model for the small transversal vibrations of elastic strings or membranes.

\paragraph{\textmd{\textit{Local and global existence results}}}

Existence of solutions to problem (\ref{K:eqn})--(\ref{K:data}) has been deeply investigated in the literature. For the sake of shortness, unless otherwise stated, here we limit ourselves to recall the main results for the case in which the nonlinearity is locally Lipschitz continuous and satisfies the strict hyperbolicity assumption
\begin{equation}
	m(\sigma)\geq\mu_{1}>0
	\quad\quad
	\forall\sigma\geq 0.
	\label{hp:m-sh}
\end{equation}

Under these assumptions, problem (\ref{K:eqn})--(\ref{K:data}) admits a local-in-time strong solution
\begin{equation}
u\in C^{0}\left([0,T],D(A^{3/4})\right)\cap C^{1}\left([0,T],D(A^{1/4})\right)
\label{defn:strong-sol}
\end{equation}
provided that
\begin{equation}
(u_{0},u_{1})\in D(A^{3/4})\times D(A^{1/4})
\qquad\text{and}\qquad
f\in C^{0}\left([0,+\infty),D(A^{1/4})\right),
\nonumber
\end{equation}
and this solution is unique in the class of strong solutions, namely solutions with the regularity (\ref{defn:strong-sol}). This result was substantially established by S.~Bernstein in the pioneering paper~\cite{1940-Bernstein}, and then refined by many authors (see~\cite{1996-TAMS-AroPan} for a modern version).

Global-in-time strong solutions are known to exist in many different special cases, which we briefly describe below.

\begin{enumerate}

\item  (Analytic case). Problem (\ref{K:eqn})--(\ref{K:data}) admits a global solution if both the initial data and the forcing term are analytic with respect to the space variables. Actually in this case it is enough to assume that the nonlinearity $m$ is just continuous and nonnegative. We refer to~\cite{1940-Bernstein,1984-RNM-AroSpa,1992-InvM-DAnSpa,1992-Ferrara-DAnSpa} for more details.

\item  (Quasi-analytic case). Problem (\ref{K:eqn})--(\ref{K:data}) admits a global solution if both the initial data and the forcing term are quasi-analytic with respect to the space variables. This is not just a refinement of the analytic case, because here the known proofs require in an essential way the Lipschitz continuity and the strict hyperbolicity of the nonlinearity (see~\cite{1984-Tokyo-Nishihara,gg:K-Nishihara}).

\item  (Special nonlinearities). In the case where $m(\sigma)=(a+b\sigma)^{-2}$ for some positive real numbers $a$ and $b$, problem (\ref{K:eqn})--(\ref{K:data}) admits a global solution provided that
\begin{equation}
(u_{0},u_{1})\in D(A)\times D(A^{1/2})
\qquad\text{and}\qquad
f\in C^{0}\left([0,+\infty),D(A)\right).
\nonumber
\end{equation}

The technical reason is that in this case (and in some sense only in this case) the equation admits a higher order quantity whose growth can be controlled for all positive times. We refer to~\cite{1985-Pohozaev} for the details.

\item  (Dispersive equations). Global existence results have been obtained in the concrete case where $A$ is the usual Laplace operator in the whole space $\re^{d}$ or in an external domain. The prototype of these results is global existence provided that the initial data and the forcing term have Sobolev regularity in the space variables, and satisfy suitable smallness assumptions and decay conditions at infinity. We refer to~\cite{1980-QAM-GreHu,1993-ARMA-DAnSpa,2005-JDE-Yamazaki,2013-JMPA-MatRuz} for precise statements.	

\item  (Spectral-gap data and operators). Global existence results are known in cases where both the initial data and the forcing term are ``lacunary'', in the sense that their spectrum contains a sequence of large ``holes''. The same is true whenever the eigenvalues of the operator $A$ are a sequence that grows fast enough. We refer to~\cite{2005-JDE-Manfrin,2006-JDE-Hirosawa,gg:K-Manfrinosawa,gg:K-Nishihara,2015-NLATMA-Hirosawa} for precise statements.	For the sake of completeness, we point out that the spectral gap theory has been recently extended in order to show the existence of global weak solutions in the energy space $D(A^{1/2})\times H$ (see~\cite{gg:EnergySpaceSG}).

\end{enumerate}

The main open problem for Kirchhoff equations is the existence of global solutions for initial data and forcing terms below the analytic or quasi-analytic regularity, for example in Gevrey spaces or in the Sobolev spaces $D(A^{\alpha})$.

\paragraph{\textmd{\textit{Simple modes}}}

Let us consider the unforced equation
\begin{equation}
u''(t)+m\left(|A^{1/2}u(t)|^{2}\right)Au(t)=0.
\label{K:eqn-f=0}
\end{equation}

If $e_{k}$ is an eigenvector of $A$ with eigenvalue $\lambda_{k}^{2}>0$, and both $u_{0}$ and $u_{1}$ are multiples of $e_{k}$, for example $u_{0}=\alpha e_{k}$ and $u_{1}=\beta e_{k}$, then the solution to (\ref{K:eqn-f=0})--(\ref{K:data}) remains a multiple of $e_{k}$ for all times, and more precisely $u(t)=z(t)e_{k}$, where $z(t)$ is the solution to the ordinary differential equation
\begin{equation}
z''(t)+\lambda_{k}^{2}m(\lambda_{k}^{2}z(t)^{2})z(t)=0
\nonumber
\end{equation}
with initial data $z(0)=\alpha$ and $z'(0)=\beta$.

These special solutions are called \emph{simple modes}, and it is well-known that they are time-periodic. Their stability has been studied deeply in the literature (see~\cite{1980-Dickey,1996-QuartAM-CazWei,gg:stability,gg:Toulouse,2003-EJDE-Ghisi}). In particular, there are many examples of unstable simple modes, and when this is the case there exist periodic trajectories that are asymptotic to them as $t\to +\infty$ or as $t\to -\infty$. What is not know yet is whether the stable manifold of a simple mode can intersect the unstable manifold of a different simple mode. This intersection would deliver a trajectory that is asymptotic,  as $t\to -\infty$ and as $t\to +\infty$, to two simple modes corresponding to different frequencies. Such a trajectory, which we call \emph{heteroclinic connection}, would realize a transfer of the energy from a low frequency to a higher frequency (due to reversibility we can always change the verse of time).

\paragraph{\textmd{\textit{Our result}}}

In this paper we \emph{assume} that, for some choice of the nonlinearity $m$, the unforced equation (\ref{K:eqn-f=0}), in the special case where $H=\re^{2}$ and $A$ is an operator with two eigenvalues equal to~$1$ and $\lambda^{2}>1$, admits a heteroclinic connection between the simple modes corresponding to the two eigenvalues (see Definition~\ref{defn:HCA}). 

Under this assumption, in Theorem~\ref{thm:main} we show that in every infinitely dimensional Hilbert space $H$ there exists an operator $A$ and an external force $f(t)$, smooth for every $t\geq 0$, for which the forced equation (\ref{K:eqn}) admits a solution that blows-up in a finite time $T_{\infty}$. By ``blow-up'' we mean that the pair $(u(t),u'(t))$ does not admit a limit as $t\to T_{\infty}^{-}$ in the energy space $D(A^{1/2})\times H$ (where the solution is necessarily bounded), while all higher order norms of the form 
\begin{equation}
|A^{\alpha}u'(t)|^{2}+|A^{\alpha+1/2}u(t)|^{2}
\qquad
\text{(with $\alpha>0$)}
\nonumber
\end{equation}
tend to $+\infty$ as $t\to T_{\infty}^{-}$.

The initial datum of this solution has only one Fourier component, and in particular it is analytic, and hence necessarily the forcing term does not lie in any analytic or quasi-analytic class, because otherwise the solution would be global. Nevertheless, the forcing term can be chosen to lie in any class that is less than quasi-analytic, and in particular in all Gevrey spaces $\mathcal{G}_{s}$ with $s>1$. In other words, the existence of a heteroclinic connection would imply that the classical global existence result in quasi-analytic classes is optimal.

\paragraph{\textmd{\textit{Overview of the technique}}}

Our proof involves three main steps.
\begin{itemize}

\item  In the first step (Proposition~\ref{prop:transition}) we consider the heteroclinic connection that we assumed to exist, and we show that for every interval $[a,b]$ we can modify it in order to obtain a new trajectory that \emph{coincides} with the first limiting simple mode for every $t\leq a$, and \emph{coincides} with the second limiting simple mode for every $t\geq b$. We think of this new trajectory as a sort of bridge that connects the two simple modes in the interval $[a,b]$. In general this bridge is not a solution of an unforced Kirchhoff equation, but rather of a Kirchhoff equation with a forcing term whose size depends on the length of the interval (the longer is the length of the bridge, the smaller is the external force required).

\item  In the second step (Proposition~\ref{prop:iteration}) we exploit a natural rescaling property of simple modes, and from the bridge between the frequencies 1 and $\lambda^{2}$ we obtain a bridge between the frequencies $\lambda^{2}$ and $\lambda^{4}$, and more generally between the frequencies $\lambda^{2k}$ and $\lambda^{2k+2}$.

\item  In the third step we consider an operator that admits the sequence $\lambda^{2k}$ among its eigenvalues, and we consider the solution obtained by glueing all the bridges. This solution moves the energy toward higher and higher frequencies. More precisely, there exists an increasing sequence $\{T_{k}\}$ of times with the property that at time $T_{k}$ the solution coincides with the simple mode corresponding to the frequency $\lambda^{2k}$. The key point is that we can arrange things in such a way that $T_{k}$ converges to some finite time $T_{\infty}$, and the size of the corresponding forcing terms vanishes as $t\to T_{\infty}^{-}$.

\end{itemize}

In summary, the existence of a heteroclinic connection between the frequencies~1 and $\lambda^{2}$ for the unforced equation (\ref{K:eqn-f=0}) implies the existence of a heteroclinic connection between any two ``consecutive'' frequencies $\lambda^{2k}$ and $\lambda^{2k+2}$, again for the unforced equation. Thanks to a suitable external force, we can switch from one connection to the next one, and in this way we obtain a trajectory that visits all frequencies. We can keep the external force small, and actually vanishing as $t\to T_{\infty}^{-}$, because the bulk of the work is done by the nonlinearity, and the only role of the external force is to put the solution on the right track from time to time.

It could be interesting to realize a similar path without the aid of the external force, but relying only on a suitable choice of initial data. This remains a challenging open problem.

\paragraph{\textmd{\textit{Consequences}}}

Our result in some sense proves nothing, because we do not know yet whether a heteroclinic connection exists or not.

If one believes that Kirchhoff equations are not well-posed in Sobolev or Gevrey spaces, we have reduced the search of a counterexample to the existence of a special trajectory for a Hamiltonian system in dimension two. We do hope that the community working on dynamical systems could contribute in this direction.

On the contrary, if one believes that Kirchhoff equations do admit global solutions for all data in Sobolev or Gevrey spaces, our result shows that the proof has to exclude the existence of heteroclinic connections, and therefore it is very likely that it has to involve some ``global'' property of the nonlinearity.

\paragraph{\textmd{\textit{Structure of the paper}}}

This paper is organized as follows. In section~\ref{sec:statements} we present formally the heteroclinic connection assumption and we state our main result. In section~\ref{sec:proofs} we prove the main result. 


\setcounter{equation}{0}
\section{Statements}\label{sec:statements}

Let us state formally the main assumption of this paper.

\begin{defn}[Heteroclinic connection assumption]\label{defn:HCA}
\begin{em}

Let $m:[0,+\infty)\to(0,+\infty)$ be a function of class $C^1$ satisfying the strict hyperbolicity assumption (\ref{hp:m-sh}), and let $\lambda>1$ be a real number. We say that the pair $(m,\lambda)$ satisfies the \emph{heteroclinic connection assumption} if there exist two functions $v:\re\to\re$ and $w:\re\to\re$ of class $C^2$ with the following properties.
\begin{itemize}

\item  They satisfy the non-triviality condition
\begin{equation}
u'(0)^{2}+v'(0)^{2}+u(0)^{2}+v(0)^{2}>0.
\label{hp:HCA-nontrivial}
\end{equation}

\item They solve the system of ordinary differential equations
\begin{equation}
\begin{cases}
v''(t)+m\left(v(t)^2+\lambda^{2}w(t)^{2}\strut\right)v(t)=0 
\\
w''(t)+\lambda^{2}m\left(v(t)^2+\lambda^{2}w(t)^{2}\strut\right)w(t)=0 
\end{cases}
\qquad
\forall t\in\re.
\label{defn:HC-system}
\end{equation}

\item There exist two positive real numbers $A_{0}$ and $B_{0}$ such that
\begin{gather}
|v'(t)|^{2}+|v(t)|^{2}\leq B_{0}\exp(-A_{0}t)
\qquad
\forall t\geq 0,
\label{hp:hc-v-decay}
\\[0.5ex]
|w'(t)|^{2}+\lambda^{2}|w(t)|^{2}\leq B_{0}\exp(-A_{0}|t|)
\qquad
\forall t\leq 0.
\label{hp:hc-w-decay}
\end{gather}

\end{itemize}
\end{em}
\end{defn}

Before moving to our main result, we need to recall the usual notion of Gevrey spaces with respect to an operator.

\begin{defn}[Gevrey spaces]
\begin{em}

Let $H$ be a real Hilbert space, let $\{e_{k}\}_{k\geq 0}\subseteq H$ be a sequence of orthonormal vectors (not necessarily a basis), let $\{\lambda_{k}\}_{k\geq 0}$ be a sequence of positive real numbers, and let $A$ be a linear operator on $H$ such that
\begin{equation}
Ae_{k}=\lambda_{k}^{2}e_{k}
\qquad
\forall k\geq 0.
\nonumber
\end{equation}

Given two positive real numbers $r$ and $s$, and a vector $u\in H$, we say that $u$ is a Gevrey vector with exponent $s$ and radius $r$, and we write $u\in\mathcal{G}_{r,s}(A)$, if
\begin{equation}
u=\sum_{k=0}^{\infty}\langle u,e_{k}\rangle e_{k}
\qquad\text{and}\qquad
\|u\|_{\mathcal{G}_{r,s}(A)}^{2}:=
\sum_{k=0}^{\infty}\langle u,e_{k}\rangle^{2}\exp\left(r\lambda_{k}^{1/s}\right)<+\infty.
\nonumber
\end{equation}

\end{em}
\end{defn} 
We recall that $\|u\|_{\mathcal{G}_{r,s}(A)}$ induces a structure of Hilbert space on the set $\mathcal{G}_{r,s}(A)$, and that the case $s=1$ corresponds to analytic vectors. More general spaces can be introduced by considering different functions of $\lambda_{k}$ in the exponential.

We are now ready to state our main result.


\begin{thm}[From a heteroclinic connection to a blow-up for a Kirchhoff equation]\label{thm:main}

Let $m:[0,+\infty)\to(0,+\infty)$ be a function of class $C^1$ satisfying the strict hyperbolicity assumption (\ref{hp:m-sh}), and let $\lambda>1$ be a real number. Let us assume that the pair $(m,\lambda)$ satisfies the heteroclinic connection assumption of Definition~\ref{defn:HCA}. 

Let $H$ be a Hilbert space, and let $A$ be a self-adjoint linear operator on $H$ for which there exists a sequence $\{e_{k}\}_{k\geq 0}\subseteq H$ of orthonormal vectors (not necessarily a basis) such that
\begin{equation}
Ae_{k}=\lambda^{2k}e_{k}
\qquad
\forall k\geq 0.
\nonumber
\end{equation}

Then there exist a function $f:[0,+\infty)\to H$ such that
\begin{equation}
f\in C^{0}\left([0,+\infty),\mathcal{G}_{s,r}(A)\right)
\qquad
\forall s>1,
\quad
\forall r>0,
\label{th:reg-f}
\end{equation}
a real number $T_{\infty}$, and a solution 
\begin{equation}
u\in C^{2}\left([0,T_{\infty}),\mathcal{G}_{s,r}(A)\right)
\qquad
\forall s>0,
\quad
\forall r>0,
\label{th:reg-u}
\end{equation}
to equation (\ref{K:eqn}) such that
\begin{equation}
\limsup_{t\to T_{\infty}^{-}}|A^{\alpha}u'(t)|^{2}+|A^{\alpha+1/2}u(t)|^{2}=+\infty
\qquad
\forall\alpha>0,
\label{th:bu-energy}
\end{equation}
and 
\begin{equation}
\lim_{t\to T_{\infty}^{-}}u'(t)
\quad
\text{does not exist}.
\label{th:no-limit}
\end{equation}

\end{thm}

We conclude with some comments on Theorem~\ref{thm:main} above.

\begin{rmk}[Time regularity]
\begin{em}

The regularity of $u$ and $f$ with respect to time depends only on the regularity of the nonlinearity $m$. If $m$ is of class $C^{1}$, then any pair of solutions $v$ and $w$ to (\ref{defn:HC-system}) is automatically of class $C^{3}$. At this point, a careful inspection of our construction reveals that actually we can improve (\ref{th:reg-f}) and (\ref{th:reg-u}), respectively, to $C^{1}$ and $C^{3}$ regularity. We spare the reader from the details, which would require only longer but standard estimates, without introducing new ideas.

In the same way, if the nonlinearity $m$ is of class $C^{r}$, then we obtain $C^{r}$ regularity in (\ref{th:reg-f}), and $C^{r+2}$ regularity in (\ref{th:reg-u}).

\end{em}
\end{rmk}

\begin{rmk}[Space regularity]
\begin{em}

We observe that in (\ref{th:reg-u}) we allow also values $s<1$, which means that the solution $u$ is more than analytic. More precisely, for every $T\in(0,T_{\infty})$ there exists a finitely dimensional subspace $H_{T}$ of $H$ such that $u(t)\in H_{T}$ for every $t\in[0,T]$.

Concerning the forcing term, the function $f$ that we construct in the proof satisfies
\begin{equation}
\sum_{k=0}^{\infty}\langle f(t),e_{k}\rangle^{2}\exp\left(\frac{c\lambda^{k}}{(k+1)^{2}}\right)<+\infty
\qquad
\forall t\geq 0
\nonumber
\end{equation}
for a suitable $c>0$. This implies the Gevrey regularity (\ref{th:reg-f}). More generally, for every increasing function $\varphi:[1,+\infty)\to[0,+\infty)$ such that
\begin{equation}
\int_{1}^{+\infty}\frac{\varphi(\sigma)}{\sigma^{2}}\,d\sigma<+\infty
\qquad\text{and}\qquad
\frac{\varphi(\sigma)}{\sigma^{2}}
\text{ is nonincreasing},
\label{defn:qa}
\end{equation}
we can modify our construction (it is enough to modify the choice of $S_{k}$ in the last step) in such a way that
\begin{equation}
\sum_{k=0}^{\infty}\langle f(t),e_{k}\rangle^{2}\exp\left(\varphi(\lambda^{2k})\right)<+\infty
\qquad
\forall t\geq 0.
\label{th:f-qa}
\end{equation}

We recall that, if $f(t)$ satisfies (\ref{th:f-qa}) for some increasing function $\varphi$ for which the integral in (\ref{defn:qa}) is divergent, then $f(t)$ is quasi-analytic.

\end{em}
\end{rmk}

\begin{rmk}[Size of the external force]
\begin{em}

A careful inspection of the proof reveals that we can choose the norm of the external force (in any fixed non quasi-analytic class) to be smaller than any given positive constant. To this end, again it is enough to modify the definition of $S_{k}$ in the last step. The other side of the coin is that the blow-up time $T_{\infty}$ depends on the norm of the external force, and tends to infinity when the norm of the external force vanishes.

\end{em}
\end{rmk}

\begin{rmk}[More general heteroclinic connections]
\begin{em}

For the sake of simplicity the heteroclinic connection assumption, as stated in Definition~\ref{defn:HCA}, involves a system that corresponds to an unforced Kirchhoff equation with two components. We can generalize the assumption by considering systems that originate from unforced Kirchhoff equations with a finite number of components. The request becomes that all components, with the exception of one, decay exponentially as $t\to -\infty$, and all components, with the exception of another one, corresponding to a different eigenvalue, decay exponentially as $t\to +\infty$.

Under this more general assumption we can again reproduce the phenomenon of Theorem~\ref{thm:main}, more or less with the same proof. We just need to be more careful in the choice of the eigenvalues of the operator, in order to reproduce the bridge at different scales.

\end{em}
\end{rmk}

\begin{rmk}[More general operators]
\begin{em}

For the sake of simplicity we decided to construct the counterexample for an operator that admits the sequence $\{\lambda^{2k}\}$ among its eigenvalues. With some extra (but rather standard) work, it is possible to extend the construction to more general multiplication operators, provided that the sequence $\{\lambda^{2k}\}$ is contained in the support of the spectrum. We point out that this is always true, for example, in the concrete case where $A$ is the Laplacian in the whole space $\re^{d}$, in which case the support of the spectrum is the half-line $[0,+\infty)$.

\end{em}
\end{rmk}

\begin{rmk}[Relation with known global existence results]
\begin{em}

We observe that our construction dodges carefully all the key assumptions of the known global existence results quoted in the introduction. Indeed, we already observed that in our construction we can not take the forcing term $f(t)$ to be analytic or quasi-analytic. Moreover, a heteroclinic connection does not exist when $m$ is the nonlinearity considered in~\cite{1985-Pohozaev} (this can be proved by using Pohozaev's invariant), or when the energy is small enough, because simple modes with small energy are known to be stable (see~\cite{1980-Dickey,gg:stability}), and therefore the existence of a heteroclinic connection does not prevent global existence for small initial data. Finally, the operator and the forcing term that we consider do not fall in the spectral gap regime, because the sequence $\lambda^{2k}$ does not grow fast enough as required by those results.

\end{em}
\end{rmk}


\setcounter{equation}{0}
\section{Proof of the main result}\label{sec:proofs}

\subsubsection*{Boundedness and Lipschitz continuity of the nonlinearity}

To begin with, we observe that the solution $(v(t),w(t))$ to system (\ref{defn:HC-system}) satisfies the classical energy equality 
\begin{equation}
v'(t)^{2}+w'(t)^{2}+M\left(v(t)^{2}+\lambda^{2}w(t)^{2}\right)=H_{0}^2
\qquad
\forall t\in\re,
\label{hamiltonian}
\end{equation}
for some real number $H_{0}>0$, where
\begin{equation}
M(\sigma):=\int_{0}^{\sigma}m(s)\,ds
\qquad
\forall\sigma\geq 0,
\label{defn:M}
\end{equation}
and the positivity of $H_{0}$ follows from (\ref{hp:HCA-nontrivial}). Due to the strict hyperbolicity assumption (\ref{hp:m-sh}), the function (\ref{defn:M}) satisfies $M(\sigma)\geq\mu_{1}\sigma$ for every $\sigma\geq 0$. Thus from (\ref{hamiltonian}) it follows that
\begin{equation}
v(t)^{2}+\lambda^{2}w(t)^{2}\leq\frac{H_{0}^{2}}{\mu_{1}}
\qquad\quad\text{and}\quad\qquad
v'(t)^{2}+w'(t)^{2}\leq H_{0}^{2}
\nonumber
\end{equation}
for every $t\in\re$, and in particular
\begin{equation}
\sup_{t\in\re}\left\{\strut
|v'(t)|,|v(t)|,|w'(t)|,\lambda|w(t)|\right\}\leq
\max\left\{1,\frac{1}{\sqrt{\mu_{1}}}\right\}H_{0}=: 
H_{1}.
\label{defn:H1}
\end{equation}

As a consequence, in the sequel we can assume that
\begin{equation}
0<\mu_{1}\leq m(\sigma)\leq\mu_{2}
\qquad
\forall\sigma\geq 0,
\nonumber
\end{equation}
and that $m$ is Lipschitz continuous with some Lipschitz constant $L$.

\subsubsection*{Simple modes}

Let $z_{1}:\re\to\re$ denote the standard simple mode with energy $H_{0}$, namely the solution to the ordinary differential equation
\begin{equation}
z_{1}''(t)+m\left(z_{1}(t)^2\strut\right)z_{1}(t)=0
\nonumber
\end{equation}
with initial data
\begin{equation}
z_{1}(0)=0,
\qquad
z_{1}'(0)=H_{0}.
\nonumber
\end{equation}

It is well-known that $z_{1}(t)$ is a periodic function of class $C^{3}$, and we call $\pi_{1}$ its minimal period. One can check that, for every real number $\lambda>0$, the function defined by
\begin{equation}
z_{\lambda}(t):=\frac{1}{\lambda}z_{1}(\lambda t)
\qquad
\forall t\in\re,
\nonumber
\end{equation}
whose minimal period is of course $\pi_{\lambda}:=\pi_{1}/\lambda$,
is a solution to equation
\begin{equation}
z_{\lambda}''(t)+\lambda^{2}m\left(\lambda^{2}z_{\lambda}(t)^2\strut\right)z_{\lambda}(t)=0
\nonumber
\end{equation}
with the same initial data
\begin{equation}
z_{\lambda}(0)=0,
\qquad
z_{\lambda}'(0)=H_{0}.
\nonumber
\end{equation}

We call $z_{\lambda}$ the simple mode of energy $H_{0}$ corresponding to the frequency $\lambda^{2}$. We observe that all simple modes satisfy the energy equality
\begin{equation}
z_{\lambda}'(t)^{2}+M\left(\lambda^{2}z_{\lambda}(t)^2\strut\right)=H_{0}^{2}
\qquad
\forall t\in\re,
\nonumber
\end{equation}
and therefore
\begin{equation}
\sup_{t\in\re}\left\{\strut 
|z_{1}'(t)|,|z_{1}(t)|,|z_{\lambda}'(t)|,\lambda|z_{\lambda}(t)|
\right\}\leq H_{1},
\label{defn:H1-bis}
\end{equation}
where $H_{1}$ is the same constant as in (\ref{defn:H1}).

The key point is that the heteroclinic connection $(v(t),w(t))$ is exponentially asymptotic to the simple mode $z_{1}$ as $t\to -\infty$, and to the simple mode $z_{\lambda}$ as $t\to +\infty$. The formal statement is the following (we omit the long but rather standard proof).

\begin{lemma}[Limiting periodic orbits]\label{lemma:limit-orbits}

Let $m:[0,+\infty)\to(0,+\infty)$ be a function of class $C^1$ satisfying the strict hyperbolicity assumption (\ref{hp:m-sh}), and let $\lambda>1$ be a real number. Let us assume that the pair $(m,\lambda)$ satisfies the heteroclinic connection assumption. 

Then there exist two real numbers $\tau_{0}\in[0,\pi_{1}]$ and $\tau_{1}\in[0,\pi_{\lambda}]$, and two positive real constants $A_{1}$ and $B_{1}$, such that
\begin{equation}
|v'(t)-z_{1}'(t-\tau_{0})|^{2}+|v(t)-z_{1}(t-\tau_{0})|^{2}\leq B_{1}\exp(-A_{1}|t|)
\qquad
\forall t\leq 0,
\label{th:v-z-decay}
\end{equation}
and
\begin{equation}
|w'(t)-z_{\lambda}'(t-\tau_{1})|^{2}+
\lambda^{2}|w(t)-z_{\lambda}(t-\tau_{1})|^{2}\leq 
B_{1}\exp(-A_{1}t)
\qquad
\forall t\geq 0.
\label{th:w-z-decay}
\end{equation}

\end{lemma}


\subsubsection*{The basic bridge between simple modes}

Thanks to Lemma~\ref{lemma:limit-orbits}, we can think of the heteroclinic connection $(v(t),w(t))$ as a trajectory that connects the two simple modes $z_{1}$ and $z_{\lambda}$ in an infinite time, without the aid of any external force. Now we show that, if we allow an external force, then we can find a trajectory $(v_{S}(t),w_{S}(t))$ that connects the same two simple modes in a finite time interval $[-2S,2S]$. The size of the required external force decays exponentially when $S$ grows.

\begin{prop}[Transition between two simple modes in a finite time interval]\label{prop:transition}

Let $m:[0,+\infty)\to(0,+\infty)$ be a function of class $C^1$ satisfying the strict hyperbolicity assumption (\ref{hp:m-sh}), and let $\lambda>1$ be a real number. Let us assume that the pair $(m,\lambda)$ satisfies the heteroclinic connection assumption. 

Then for every $S>0$ there exist two functions $v_{S}:\re\to\re$ and $w_{S}:\re\to\re$ of class $C^{2}$, and two continuous functions $\varphi_{S}:\re\to\re$ and $\psi_{S}:\re\to\re$, with the following properties.

\begin{enumerate}
\renewcommand{\labelenumi}{(\arabic{enumi})}

\item \emph{(Kirchhoff equation with two modes and forcing term).} The functions $v_{S}$ and $w_{S}$ are solutions to the system
\begin{equation}
\begin{cases}
v_{S}''(t)+m\left(v_{S}(t)^2+\lambda^{2}w_{S}(t)^{2}\strut\right)v_{S}(t)=\varphi_{S}(t) 
\\[0.5ex]
w_{S}''(t)+\lambda^{2}m\left(v_{S}(t)^2+\lambda^{2}w_{S}(t)^{2}\strut\right)w_{S}(t)=\psi_{S}(t) 
\end{cases}
\qquad
\forall t\in\re.
\label{th:eqn-vSwS}
\end{equation}

\item \emph{(Conditions at infinity).} The functions $v_{S}$ and $w_{S}$ satisfy
\begin{equation}
v_{S}(t)=z_{1}(t-\tau_{0})
\quad\text{and}\quad
w_{S}(t)=0
\qquad
\forall t\leq -2S,
\label{th:v-w-neg}
\end{equation}
and
\begin{equation}
v_{S}(t)=0
\quad\text{and}\quad
w_{S}(t)=z_{\lambda}(t-\tau_{1})
\qquad
\forall t\geq 2S.
\label{th:v-w-pos}
\end{equation}

In particular, there exist $S_{1}\in[2S,2S+\pi_{1}]$ and $S_{2}\in[2S,2S+\pi_{\lambda}]$ such that
\begin{equation}
w_{S}(-S_{1})=w_{S}'(-S_{1})=v_{S}(-S_{1})=0,
\qquad
v_{S}'(-S_{1})=H_{0},
\nonumber
\end{equation}
and
\begin{equation}
v_{S}(S_{2})=v_{S}'(S_{2})=w_{S}(S_{2})=0,
\qquad
w_{S}'(S_{2})=H_{0}.
\nonumber
\end{equation}

\item  \emph{(Bound on the forcing term).} The functions $\varphi_{S}$ and $\psi_{S}$ are such that 
\begin{equation}
\varphi_{S}(t)=\psi_{S}(t)=0
\qquad
\forall t\in(-\infty,-2S]\cup[-S,S]\cup[2S,+\infty),
\label{th:phi-psi-0}
\end{equation}
and satisfy the estimate
\begin{equation}
|\varphi_{S}(t)|^{2}+|\psi_{S}(t)|^{2}\leq \left(\frac{1}{S^{2}}+1\right)^{2}B_{2}\exp(-A_{2}S)
\qquad
\forall t\in\re,
\label{th:est-phi-psi}
\end{equation}
where $A_{2}$ and $B_{2}$ are two positive real numbers, both independent of $S$ and $t$.

\end{enumerate}

\end{prop}

\begin{proof}

Let $\theta\in C^\infty(\re)$ be a cutoff function such that
\begin{itemize}

\item  $\theta(x)=1$ for every $x\leq 1$,

\item  $\theta(x)=0$ for every $x\geq 2$,

\item  $0\leq\theta(x)\leq 1$ for every $x\in[1,2]$,

\end{itemize}
and let $\Gamma$ be a constant such that
\begin{equation}
|\theta'(x)|+|\theta''(x)|\leq \Gamma
\qquad
\forall x\in\re.
\label{defn:bound-theta}
\end{equation}

The idea is to use the function $\theta$ in order to define $v_{S}(t)$ and $w_{S}(t)$ as a convex combination of trajectories that coincides
\begin{itemize}

\item  with the heteroclinic connection $(v(t),w(t))$ for $t\in[-S,S]$,

\item  with the limiting periodic trajectory $(z_{1}(t-\tau_{0}),0)$ for $t\leq -2S$,

\item  with the limiting periodic trajectory $(0,z_{\lambda}(t-\tau_{1}))$ for $t\geq 2S$.

\end{itemize}

\paragraph{\textit{\textmd{Definition when $t\leq 0$}}}

In the case $t\leq 0$ we set $\theta_{S}(t):=\theta(-t/S)$ and we consider the functions
\begin{equation}
v_{S}(t):=\theta_{S}(t)v(t)+(1-\theta_{S}(t))z_{1}(t-\tau_{0})
\qquad\text{and}\qquad
w_{S}(t):=\theta_{S}(t)w(t).
\nonumber
\end{equation}

Since 
\begin{equation}
\theta_{S}(t)=0
\quad
\forall t\leq -2S
\qquad\quad\text{and}\quad\qquad
\theta_{S}(t)=1
\quad
\forall t\in[-S,0],
\label{eqn:theta-S}
\end{equation}
we deduce that (\ref{th:v-w-neg}) holds true, and in addition
\begin{equation}
(v_{S}(t),w_{S}(t))=(v(t),w(t))
\qquad
\forall t\in[-S,0].
\nonumber
\end{equation}

Computing the second order time derivatives of $v_{S}$ and $w_{S}$, we discover that for $t\leq 0$ these functions are solutions to system (\ref{th:eqn-vSwS}) provided that we set
\begin{eqnarray*}
\varphi_{S}(t)  & :=  &
\theta_{S}''(t)\left\{v(t)-z_{1}(t-\tau_{0})\right\}
+2\theta_{S}'(t)\left\{v'(t)-z_{1}'(t-\tau_{0})\right\}
\\[0.5ex]
& &
\mbox{}+\theta_{S}(t)v(t)\left\{m\left(v_{S}(t)^{2}+\lambda^{2}w_{S}(t)^{2}\right)-
m\left(v(t)^{2}+\lambda^{2}w(t)^{2}\right)\right\}
\\[0.5ex]
& &
\mbox{}+(1-\theta_{S}(t))z_{1}(t-\tau_{0})\left\{m\left(v_{S}(t)^{2}+\lambda^{2}w_{S}(t)^{2}\right)-
m\left(z_{1}(t-\tau_{0})^{2}\right)\right\},
\end{eqnarray*}
and
\begin{eqnarray*}
\psi_{S}(t)  & :=  &
\theta_{S}''(t)w(t)+2\theta_{S}'(t)w'(t)
\\[0.5ex]
& &
\mbox{}+\lambda^{2}\theta_{S}(t)w(t)\left\{m\left(v_{S}(t)^{2}+\lambda^{2}w_{S}(t)^{2}\right)-
m\left(v(t)^{2}+\lambda^{2}w(t)^{2}\right)\right\}.
\end{eqnarray*}

Using again (\ref{eqn:theta-S}) we obtain that
\begin{equation}
\varphi_{S}(t)=\psi_{S}(t)=0
\qquad
\forall t\in(-\infty,-2S]\cup[-S,0],
\nonumber
\end{equation}
which proves (\ref{th:phi-psi-0}) for $t\leq 0$.

It remains to prove (\ref{th:est-phi-psi}) for $t\leq 0$. Let $L_{1}(t)$, $L_{2}(t)$, $L_{3}(t)$ denote the three lines in the definition of $\varphi_{S}(t)$. From (\ref{defn:bound-theta}) we deduce that
\begin{equation}
|\theta_{S}'(t)|\leq\frac{\Gamma}{S}
\qquad\text{and}\qquad
|\theta_{S}''(t)|\leq\frac{\Gamma}{S^{2}}
\qquad
\forall t\leq 0,
\label{est:theta''}
\end{equation}
and therefore
\begin{equation}
|L_{1}(t)|\leq
\frac{\Gamma}{S^{2}}|v(t)-z_{1}(t-\tau_{0})|+\frac{2\Gamma}{S}|v'(t)-z_{1}'(t-\tau_{0})|.
\nonumber
\end{equation}

Moreover, from (\ref{defn:H1}) and (\ref{defn:H1-bis}) we deduce that
\begin{equation}
|v_{S}(t)|\leq H_{1}
\qquad\text{and}\qquad
\lambda|w_{S}(t)|\leq H_{1}
\qquad
\forall t\leq 0,
\nonumber
\end{equation}
and therefore from the Lipschitz continuity of $m$ we obtain that
\begin{eqnarray*}
\lefteqn{\hspace{-1em}
\left|m\left(v_{S}(t)^{2}+\lambda^{2}w_{S}(t)^{2}\right)-
m\left(v(t)^{2}+\lambda^{2}w(t)^{2}\right)\right|}
\\[0.5ex]
& \leq &
L\left(\left|v_{S}(t)^{2}-v(t)^{2}\right|+\lambda^{2}\left|w_{S}(t)^{2}-w(t)^{2}\right|\right)
\\[0.5ex]
& \leq &
L\left(\left|v_{S}(t)+v(t)\right|\cdot\left|v_{S}(t)-v(t)\right|+
\left|\lambda w_{S}(t)+\lambda w(t)\right|\cdot\lambda\left|w_{S}(t)-w(t)\right|\strut\right)
\\[0.5ex]
& \leq &
2H_{1}L\left(\left|v_{S}(t)-v(t)\right|+\lambda\left|w_{S}(t)-w(t)\right|\strut\right).
\end{eqnarray*}

Finally we observe that
\begin{equation}
|v_{S}(t)-v(t)|=(1-\theta_{S}(t))|v(t)-z_{1}(t-\tau_{0})|
\nonumber
\end{equation}
and
\begin{equation}
|w_{S}(t)-w(t)|=(1-\theta_{S}(t))|w(t)|,
\nonumber
\end{equation}
so that in conclusion
\begin{multline}
\qquad
\left|m\left(v_{S}(t)^{2}+\lambda^{2}w_{S}(t)^{2}\right)-
m\left(v(t)^{2}+\lambda^{2}w(t)^{2}\right)\right|
\\
\leq
2H_{1}L\left(|v(t)-z_{1}(t-\tau_{0})|+\lambda|w(t)|\right),
\qquad
\label{est:m-m}
\end{multline}
and therefore
\begin{equation}
|L_{2}(t)|\leq 2H_{1}^{2}L\left\{
\left|v(t)-z_{1}(t-\tau_{0})\right|+\lambda\left|w(t)\right|
\strut\right\}.
\nonumber
\end{equation}

In an analogous way we obtain that
\begin{eqnarray*}
\lefteqn{\hspace{-1em}
\left|m\left(v_{S}(t)^{2}+\lambda^{2}w_{S}(t)^{2}\right)-
m\left(z_{1}(t-\tau_{0})^{2}\right)\right|}
\\[0.5ex]
& \leq &
L\left(\left|v_{S}(t)^{2}-z_{1}(t-\tau_{0})^{2}\right|+\lambda^{2}w_{S}(t)^{2}\right)
\\[0.5ex]
& \leq &
L\left(\left|v_{S}(t)+z_{1}(t-\tau_{0})\right|\cdot\left|v_{S}(t)-z_{1}(t-\tau_{0})\right|+
\lambda|w_{S}(t)|\cdot\lambda|w_{S}(t)|\strut\right)
\\[0.5ex]
& \leq &
2H_{1}L\left(\left|v_{S}(t)-z_{1}(t-\tau_{0})\right|+\lambda|w_{S}(t)|\strut\right)
\\[0.5ex]
& \leq &
2H_{1}L\left(\left|v(t)-z_{1}(t-\tau_{0})\right|+\lambda|w(t)|\strut\right),
\end{eqnarray*}
and therefore
\begin{equation}
|L_{3}(t)|\leq
2H_{1}^{2}L\left\{\left|v(t)-z_{1}(t-\tau_{0})\right|+\lambda\left|w(t)\right|\strut\right\}.
\nonumber
\end{equation}

From all these estimate we deduce that
\begin{equation}
|\varphi_{S}(t)|\leq
\left(\frac{\Gamma}{S^{2}}+4H_{1}^{2}L\right)|v(t)-z_{1}(t-\tau_{0})|
+\frac{2\Gamma}{S}|v'(t)-z_{1}'(t-\tau_{0})|+
4H_{1}^{2}L\lambda\left|w(t)\right|.
\nonumber
\end{equation}

Similarly, from (\ref{est:theta''}) and (\ref{est:m-m}) we obtain that
\begin{eqnarray*}
|\psi_{S}(t)| & \leq &
\frac{\Gamma}{S^{2}}|w(t)|+\frac{2\Gamma}{S}|w'(t)|+2H_{1}^{2}L\lambda
\left\{|v(t)-z_{1}(t-t_{0})|+\lambda|w(t)|\right\}
\\
& \leq &
\left(\frac{\Gamma}{S^{2}\lambda}+2H_{1}^{2}L\lambda\right)\lambda|w(t)|
+\frac{2\Gamma}{S}|w'(t)|+2H_{1}^{2}L\lambda|v(t)-z_{1}(t-t_{0})|.
\end{eqnarray*}

Finally, taking (\ref{hp:hc-w-decay}) and (\ref{th:v-z-decay}) into account, we conclude that
\begin{equation}
|\varphi_{S}(t)|^{2}\leq
2\left(\frac{\Gamma}{S^{2}}+\frac{2\Gamma}{S}+4H_{1}^{2}L\right)^{2}B_{1}\exp\left(-A_{1}|t|\right)
+32H_{1}^{4}L^{2}B_{0}\exp\left(-A_{0}|t|\right),
\nonumber
\end{equation}
and analogously
\begin{equation}
|\psi_{S}(t)|^{2}\leq
2\left(\frac{\Gamma}{S^{2}\lambda}+\frac{2\Gamma}{S}+2H_{1}^{2}L\lambda\right)^{2}B_{0}\exp\left(-A_{0}|t|\right)
+8H_{1}^{4}L^{2}\lambda^{2}B_{1}\exp\left(-A_{1}|t|\right).
\nonumber
\end{equation}

Recalling that $\varphi_{S}(t)=\psi_{S}(t)=0$ for $t\in[-S,0]$, the last two inequalities imply (\ref{th:est-phi-psi}) for $t\leq 0$.

\paragraph{\textit{\textmd{Definition when $t\geq 0$}}}

In the case $t\geq 0$ we set $\theta_{S}(t):=\theta(t/S)$ and we consider the functions
\begin{equation}
v_{S}(t):=\theta_{S}(t)v(t)
\qquad\text{and}\qquad
w_{S}(t):=\theta_{S}(t)w(t)+(1-\theta_{S}(t))z_{\lambda}(t-\tau_{1}).
\nonumber
\end{equation}

As in the previous case we find that
\begin{equation}
(v_{S}(t),w_{S}(t))=(v(t),w(t))
\qquad
\forall t\in[0,S],
\nonumber
\end{equation}
and that for $t\geq 0$ these functions are solutions to system (\ref{th:eqn-vSwS}) provided that we set
\begin{eqnarray*}
\varphi_{S}(t)  & :=  &
\theta_{S}''(t)v(t)+2\theta_{S}'(t)v'(t)
\\[0.5ex]
& &
\mbox{}+\theta_{S}(t)v(t)\left\{m\left(v_{S}(t)^{2}+\lambda^{2}w_{S}(t)^{2}\right)-
m\left(v(t)^{2}+\lambda^{2}w(t)^{2}\right)\right\}.
\end{eqnarray*}
and
\begin{eqnarray*}
\psi_{S}(t)  & :=  &
\theta_{S}''(t)\left\{w(t)-z_{\lambda}(t-\tau_{1})\right\}
+2\theta_{S}'(t)\left\{w'(t)-z_{\lambda}'(t-\tau_{1})\right\}
\\[0.5ex]
& &
\mbox{}+\lambda^{2}\theta_{S}(t)w(t)\left\{m\left(v_{S}(t)^{2}+\lambda^{2}w_{S}(t)^{2}\right)-
m\left(v(t)^{2}+\lambda^{2}w(t)^{2}\right)\right\}
\\[0.5ex]
& &
\mbox{}+\lambda^{2}(1-\theta_{S}(t))z_{\lambda}(t-\tau_{1})\left\{m\left(v_{S}(t)^{2}+\lambda^{2}w_{S}(t)^{2}\right)-
m\left(\lambda^{2}z_{\lambda}(t-\tau_{1})^{2}\right)\right\}.
\end{eqnarray*}

As in the previous case we obtain that
\begin{multline*}
\qquad
\left|m\left(v_{S}(t)^{2}+\lambda^{2}w_{S}(t)^{2}\right)-
m\left(v(t)^{2}+\lambda^{2}w(t)^{2}\right)\right|
\\
\leq
2H_{1}L\left(|v(t)|+\lambda|w(t)-z_{\lambda}(t-\tau_{1})|\strut\right),
\qquad
\end{multline*}
and
\begin{multline*}
\qquad
\left|m\left(v_{S}(t)^{2}+\lambda^{2}w_{S}(t)^{2}\right)-
m\left(\lambda^{2}z_{\lambda}(t-\tau_{1})^{2}\right)\right|
\\
\leq
2H_{1}L\left(|v(t)|+\lambda|w(t)-z_{\lambda}(t-\tau_{1})|\strut\right),
\qquad
\end{multline*}
from which we deduce that
\begin{equation}
|\varphi_{S}(t)|\leq
\left(\frac{\Gamma}{S^{2}}+2H_{1}^{2}L\right)|v(t)|
+\frac{2\Gamma}{S}|v'(t)|
+2H_{1}^{2}L\lambda|w(t)-z_{\lambda}(t-\tau_{1})|
\nonumber
\end{equation}
and
\begin{equation}
|\psi_{S}(t)|\leq
\left(\frac{\Gamma}{S^{2}\lambda}+4H_{1}^{2}L\lambda\right)\lambda|w(t)-z_{\lambda}(t-\tau_{1})|
+\frac{2\Gamma}{S}|w'(t)-z_{\lambda}'(t-\tau_{1})|
+4H_{1}^{2}L\lambda|v(t)|.
\nonumber
\end{equation}

Finally, taking (\ref{hp:hc-v-decay}) and (\ref{th:w-z-decay}) into account, we conclude that
\begin{equation}
|\varphi_{S}(t)|^{2}\leq
2\left(\frac{\Gamma}{S^{2}}+\frac{2\Gamma}{S}+2H_{1}^{2}L\right)^{2}B_{0}\exp\left(-A_{0}t\right)
+8H_{1}^{4}L^{2}B_{1}\exp\left(-A_{1}t\right),
\nonumber
\end{equation}
and
\begin{equation}
|\psi_{S}(t)|^{2}\leq
2\left(\frac{\Gamma}{S^{2}\lambda}+\frac{2\Gamma}{S}+4H_{1}^{2}L\lambda\right)^{2}B_{1}\exp\left(-A_{1}t\right)
+32H_{1}^{4}L^{2}\lambda^{2}B_{0}\exp\left(-A_{0}t\right).
\nonumber
\end{equation}

Recalling that $\varphi_{S}(t)=\psi_{S}(t)=0$ for $t\in[0,S]$, the last two inequalities imply (\ref{th:est-phi-psi}) for $t\geq 0$.
\end{proof}


\subsubsection*{A sequence of bridges between consecutive simple modes}

In the next result we rescale the construction of Proposition~\ref{prop:transition}, and we obtain a bridge between the simple modes corresponding to the frequencies $\lambda^{2k}$ and $\lambda^{2k+2}$.

\begin{prop}[Rescaling]\label{prop:iteration}

Let $m:[0,+\infty)\to(0,+\infty)$ be a function of class $C^1$ satisfying the strict hyperbolicity assumption (\ref{hp:m-sh}), and let $\lambda>1$ be a real number. Let us assume that the pair $(m,\lambda)$ satisfies the heteroclinic connection assumption. 

Let $H$ be a Hilbert space, and let $A$ be an operator as in Theorem~\ref{thm:main}. Let $k$ be a positive integer, and let $S_{k}$ be a positive real number.

Then there exist two functions $u_{k}:\re\to H$ and $f_{k}:\re\to H$ with the following properties.
\begin{enumerate}
\renewcommand{\labelenumi}{(\arabic{enumi})}

\item \emph{(Regularity).} The functions $u_{k}$ and $f_{k}$ satisfy
\begin{equation}
u_{k}\in C^{2}\left(\re,\Span(e_{k},e_{k+1})\right)
\qquad\text{and}\qquad
f_{k}\in C^{0}\left(\re,\Span(e_{k},e_{k+1})\right).
\label{th:reg-uk-fk}
\end{equation} 

\item \emph{(Kirchhoff equation with two modes and forcing term).} The functions $u_{k}$ is a solutions to the forced Kirchhoff equation
\begin{equation}
u_{k}''(t)+m\left(|A^{1/2}u_{k}(t)|^2\strut\right)Au_{k}(t)=f_{k}(t) 
\qquad
\forall t\in\re.
\label{th:eqn-uk-fk}
\end{equation}

\item \emph{(Conditions at infinity).} The functions $u_{k}$ satisfies
\begin{equation}
u_{k}(t)=z_{\lambda^{k}}\left(t-\frac{\tau_{0}}{\lambda^{k}}\right)e_{k}
\qquad
\forall t\leq -2S_{k},
\label{th:uk-neg}
\end{equation}
and
\begin{equation}
u_{k}(t)=z_{\lambda^{k+1}}\left(t-\frac{\tau_{1}}{\lambda^{k}}\right)e_{k+1}
\qquad
\forall t\geq 2S_{k}.
\label{th:uk-pos}
\end{equation}

In particular, there exist $S_{1,k}\in[2S_{k},2S_{k}+\pi_{\lambda^{k}}]$ and $S_{2,k}\in[2S_{k},2S_{k}+\pi_{\lambda^{k+1}}]$ such that
\begin{equation}
u_{k}(-S_{1,k})=0,
\qquad
u_{k}'(-S_{1,k})=H_{0}e_{k},
\label{th:uk-S1k}
\end{equation}
and
\begin{equation}
u_{k}(S_{2,k})=0,
\qquad
u_{k}'(S_{2,k})=H_{0}e_{k+1}.
\label{th:uk-S2k}
\end{equation}

\item  \emph{(Properties of the forcing term).} The function $f_{k}$ is such that 
\begin{equation}
f_{k}(t)=0
\qquad
\forall t\in(-\infty,-2S_{k}]\cup[-S_{k},S_{k}]\cup[2S_{k},+\infty),
\label{th:fk=0}
\end{equation}
and for every positive value of $r$ and $s$ its norm in the Gevrey space $\mathcal{G}_{r,s}(A)$ satisfies
\begin{equation}
\|f_{k}(t)\|_{\mathcal{G}_{r,s}(A)}^{2}\leq
\left(\frac{1}{\lambda^{k}S_{k}^{2}}+\lambda^{k}\right)^{2}B_{2}
\exp\left(r\lambda^{(k+1)/s}-A_{2}\lambda^{k}S_{k}\right)
\qquad
\forall t\in\re,
\label{th:est-fk}
\end{equation}
where $A_{2}$ and $B_{2}$ are the constants that appear in (\ref{th:est-phi-psi}).

\end{enumerate}

\end{prop}

\begin{proof}

Let us apply Proposition~\ref{prop:transition} with
\begin{equation}
S:=\lambda^{k}S_{k},
\nonumber
\end{equation}
and let $v_{S}$, $w_{S}$, $\varphi_{S}$ and $\psi_{S}$ be the functions that we obtain. At this point let us set
\begin{equation}
u_{k}(t):=\frac{1}{\lambda^{k}}v_{S}(\lambda^{k}t)e_{k}+\frac{1}{\lambda^{k}}w_{S}(\lambda^{k}t)e_{k+1}
\nonumber
\end{equation}
and
\begin{equation}
f_{k}(t):=\lambda^{k}\varphi_{S}(\lambda^{k}t)e_{k}+\lambda^{k}\psi_{S}(\lambda^{k}t)e_{k+1}.
\label{defn:fk}
\end{equation}

Both the regularity (\ref{th:reg-uk-fk}) and equation (\ref{th:eqn-uk-fk}) follow from these definitions and from the corresponding properties of $v_{S}$, $w_{S}$, $\varphi_{S}$ and $\psi_{S}$.

Moreover, for every $t\leq -2S_{k}$ it turns out that $\lambda^{k}t\leq -2S$, and therefore from (\ref{th:v-w-neg}) we obtain that
\begin{equation}
u_{k}(t)=
\frac{1}{\lambda^{k}}z_{1}\left(\lambda^{k}t-\tau_{0}\right)e_{k}=
\frac{1}{\lambda^{k}}z_{1}\left(\lambda^{k}\left(t-\frac{\tau_{0}}{\lambda_{k}}\right)\right)e_{k}=
z_{\lambda^{k}}\left(t-\frac{\tau_{0}}{\lambda^{k}}\right)e_{k},
\nonumber
\end{equation}
which proves (\ref{th:uk-neg}). Similarly, for every $t\geq 2S_{k}$ it turns out that $\lambda^{k}t\geq 2S$, and therefore from (\ref{th:v-w-pos}) we obtain that
\begin{equation}
u_{k}(t)=
\frac{1}{\lambda^{k}}z_{\lambda}\left(\lambda^{k}t-\tau_{1}\right)e_{k+1}=
\frac{1}{\lambda^{k+1}}z_{1}\left(\lambda^{k+1}t-\lambda\tau_{1}\right)e_{k+1}=
z_{\lambda^{k+1}}\left(t-\frac{\tau_{1}}{\lambda^{k}}\right)e_{k+1},
\nonumber
\end{equation}
which proves (\ref{th:uk-pos}).

Finally, (\ref{th:fk=0}) follows from (\ref{th:phi-psi-0}), while from (\ref{defn:fk}) and (\ref{th:est-phi-psi}) we obtain that
\begin{eqnarray*}
\|f_{k}(t)\|_{\mathcal{G}_{s,r}(A)}^{2} & \leq &
\lambda^{2k}|\varphi_{S}(\lambda^{k}t)|^{2}\exp\left(r\lambda^{k/s}\right)
+\lambda^{2k}|\psi_{S}(\lambda^{k}t)|^{2}\exp\left(r\lambda^{(k+1)/s}\right)
\\[0.5ex]
& \leq &
\left(\frac{1}{\lambda^{k}S_{k}^{2}}+\lambda^{k}\right)^{2}B_{2}
\exp\left(-A_{2}\lambda^{k}S_{k}\right)\exp\left(r\lambda^{(k+1)/s}\right)
\end{eqnarray*}
for every $t\in\re$, which proves (\ref{th:est-fk}).
\end{proof}


\subsubsection*{Conclusion of the proof of Theorem~\ref{thm:main}}

\paragraph{\textmd{\textit{Definitions}}}

For every integer $k\geq 0$ we apply Proposition~\ref{prop:iteration} with
\begin{equation}
S_{k}:=\frac{1}{(k+1)^{2}}.
\nonumber
\end{equation}

We consider the functions $u_{k}$ and $f_{k}$, and the times $S_{1,k}$ and $S_{2,k}$, provided by that result. We consider the sequence $\{T_{k}\}\subseteq[0,+\infty)$ defined by $T_{0}:=0$ and
\begin{equation}
T_{k+1}:=T_{k}+S_{1,k}+S_{2,k}
\qquad
\forall k\geq 0.
\nonumber
\end{equation}

We observe that
\begin{equation}
S_{1,k}\in [2S_{k},2S_{k}+\pi_{\lambda^{k}}]
\qquad\text{and}\qquad
S_{2,k}\in[2S_{k},2S_{k}+\pi_{\lambda^{k+1}}],
\nonumber
\end{equation}
and therefore
\begin{equation}
T_{k+1}-T_{k}=
S_{1,k}+S_{2,k}\leq
\frac{4}{(k+1)^{2}}+\pi_{\lambda^{k}}+\pi_{\lambda^{k+1}}=
\frac{4}{(k+1)^{2}}+\pi_{1}\left(\frac{1}{\lambda^{k}}+\frac{1}{\lambda^{k+1}}\right).
\nonumber
\end{equation}

As a consequence, the sequence $\{T_{k}\}\subseteq[0,+\infty)$ is increasing and 
\begin{equation}
T_{\infty}:=\lim_{k\to +\infty}T_{k}<+\infty.
\nonumber
\end{equation}

Let us define $u:[0,T_{\infty})\to H$ by
\begin{equation}
u(t):=u_{k}(t-T_{k}-S_{1,k})
\qquad
 \text{if }t\in[T_{k},T_{k+1}]\text{ for some }k\geq 0,
\nonumber
\end{equation}
and let us define $f:[0,+\infty)\to H$ by
\begin{equation}
f(t):=\begin{cases}
f_{k}(t-T_{k}-S_{1,k})      & \text{if }t\in[T_{k},T_{k+1}]\text{ for some }k\geq 0, 
\\[0.5ex]
0      & \text{if }t\geq T_{\infty}.
\end{cases}
\nonumber
\end{equation}

We claim that the conclusions of Theorem~\ref{thm:main} hold true with these choices.

\paragraph{\textmd{\textit{Regularity of $f$}}}

To begin with, we observe that the definition is well-posed because from (\ref{th:fk=0}) we know that
\begin{equation}
f_{k}(T_{k+1}-T_{k}-S_{1,k})=
f_{k}(S_{2,k})=
0=
f_{k+1}(-S_{1,k+1})=
f_{k+1}(T_{k+1}-T_{k+1}-S_{1,k+1}).
\nonumber
\end{equation}

Moreover, from (\ref{th:est-fk}) we deduce that
\begin{equation}
\|f_{k}(t)\|_{\mathcal{G}_{r,s}(A)}^{2}\leq
\left(\frac{(k+1)^{4}}{\lambda^{k}}+\lambda^{k}\right)^{2}B_{2}
\exp\left(r\lambda^{(k+1)/s}-A_{2}\frac{\lambda^{k}}{(k+1)^{2}}\right)
\qquad
\forall t\in[T_{k},T_{k+1}]
\nonumber
\end{equation}
for every positive value of $s$ and $r$. This implies in particular that
\begin{equation}
\forall s>1
\quad
\forall r>0
\qquad
\lim_{t\to T_{\infty}^{-}}f(t)=0
\quad
\text{in }\mathcal{G}_{s,r}(A),
\nonumber
\end{equation}
which shows that $f$ has the regularity (\ref{th:reg-f}).

\paragraph{\textmd{\textit{Regularity of $u$}}}

To begin with, we observe that the definition is well-posed because from (\ref{th:uk-S2k}) with $k$ and (\ref{th:uk-S1k}) with $k+1$ we know that
\begin{equation}
u_{k}(T_{k+1}-T_{k}-S_{1,k})=
u_{k}(S_{2,k})=
0=
u_{k+1}(-S_{1,k+1})=
u_{k+1}(T_{k+1}-T_{k+1}-S_{1,k+1}),
\nonumber
\end{equation}
and
\begin{equation}
u_{k}'(T_{k+1}-T_{k}-S_{1,k})=
u_{k}'(S_{2,k})=
H_{0}e_{k+1}=
u_{k+1}'(-S_{1,k+1})=
u_{k+1}'(T_{k+1}-T_{k+1}-S_{1,k+1}).
\nonumber
\end{equation}

In the same way from (\ref{th:eqn-uk-fk}) and (\ref{th:fk=0}) we obtain that
\begin{equation}
u_{k}''(T_{k+1}-T_{k}-S_{1,k})=0=u_{k+1}''(T_{k+1}-T_{k+1}-S_{1,k+1}),
\nonumber
\end{equation}
and therefore the regularity of $u$ follows from the regularity of $u_{k}$.

\paragraph{\textmd{\textit{Blow-up of $u$}}}

To this end, it is enough to observe that
\begin{eqnarray*}
\limsup_{t\to T_{\infty}^{-}}
\left(|A^{\alpha}u'(t)|^{2}+|A^{\alpha+1/2}u(t)|^{2}\right)
& \geq &
\limsup_{k\to +\infty}|A^{\alpha}u'(T_{k})|
\\
& = &
\limsup_{k\to +\infty}|A^{\alpha}(H_{0}e_{k})|^{2}
\\
& = &
\limsup_{k\to +\infty}H_{0}^{2}\lambda^{4k\alpha}
\\
& = & 
+\infty
\end{eqnarray*}
for every $\alpha>0$, which proves (\ref{th:bu-energy}).

In the same way the sequence $u'(T_{k})=H_{0}e_{k}$ has no limit as $k\to +\infty$, and therefore the limit of $u'(t)$ as $t\to T_{\infty}^{-}$ does not exist, which proves (\ref{th:no-limit}).
\qed


\subsubsection*{\centering Acknowledgments}

Both authors are members of the Italian {\selectlanguage{italian}%
``Gruppo Nazionale per l'Analisi Matematica, la Probabilit\`{a} e le loro Applicazioni'' (GNAMPA) of the ``Istituto Nazionale di Alta Matematica'' (INdAM)}. The first author was partially supported by PRIN 2020XB3EFL, ``Hamiltonian and Dispersive PDEs''.

\selectlanguage{english}



\label{NumeroPagine}

\end{document}